\documentclass[11pt]{article}
\usepackage{geometry}
\usepackage{amsmath,amssymb,amsthm,amsfonts,comment}
\usepackage{pgf,tikz}
\usetikzlibrary{arrows}
\usepackage{mathrsfs}
\usepackage{scrextend}
\usepackage{mathtools}
\usepackage[font=small,labelfont=small]{caption}
\usepackage{mathtools}

\newtheorem{nummer}{ }[section]
\newtheorem{thm}[nummer]{\sc Theorem}
\newtheorem{prop}[nummer]{\sc Proposition}
\newtheorem{lem}[nummer]{\sc Lemma}
\newtheorem{cor}[nummer]{\sc Corollary}

\newtheorem{defi}[nummer]{\sc Definition}

\newtheorem{remark}[nummer]{\sc Remark}

\renewcommand{\qedsymbol}{{\large$\boldsymbol{\dashv}$}}

\DeclareMathOperator{\fix}{{\rm fix}}

\newcommand{\fin}{\operatorname{fin}}

\begin{document}

\title{\textbf{Some implications of Ramsey Choice for $n$-element sets}}
\author{Lorenz Halbeisen and Salome Schumacher}
\date{}
\maketitle

\begin{addmargin}[5ex]{6ex}
\small \textit{key-words:} Axiom of Choice, weak forms of the Axiom of Choice, Ramsey Choice, Partial Choice for infinite families of $n$-element sets, Permutation models, Pincus' transfer theorems
\end{addmargin}

\hspace{5ex}{\small{\it 2010 Mathematics Subject Classification\/}: 03E25, 03E35 

\begin{abstract}\noindent
Let $n\in\omega$. The weak choice principle $\operatorname{RC}_n$ states that for every infinite set $x$ there is an infinite subset $y\subseteq x$ with a choice function on $[y]^n:=\{z\subseteq y\mid \lvert z\rvert =n\}$. $\operatorname{C}_n^-$ states that for every infinite family of $n$-element sets, there is an infinite subfamily $\mathcal{G}\subseteq\mathcal{F}$ with a choice function. $\operatorname{LOC}_n^-$ and $\operatorname{WOC}_n^-$ are the same statement but we assume that the family $\mathcal{F}$ is linearly orderable ($\operatorname{LOC}_n^-$) or well-orderable ($\operatorname{WOC}_n^-$).\newline

\noindent In the first part of this paper we will give a full characterization of when the implication $\operatorname{RC}_m\Rightarrow \operatorname{WOC}_n^-$ with $m,n\in\omega$ holds in $\operatorname{ZF}$. We will prove the independence results by using suitable Fraenkel-Mostowski permutation models. In the second part of we will show some generalizations. In particular we will show that $\operatorname{RC}_5\Rightarrow \operatorname{LOC}_5^-$ and $\operatorname{RC}_6\Rightarrow \operatorname{C}_3^-$, answering two open questions from Halbeisen and Tachtsis in \cite{Halbeisen17}. Furthermore we will show that $\operatorname{RC}_6\Rightarrow \operatorname{C}_9^-$ and $\operatorname{RC}_7\Rightarrow \operatorname{LOC}_7^-$.

\end{abstract}

\section{Definitions and Introduction}

The notation we use is standard and follows that of \cite{Halbeisen12}. Now we list some definitions that shall be used in the sequel: 
\begin{defi}Let $n$ be a natural number.
\begin{enumerate}
	\item $\operatorname{C}_n^-$ states that every infinite family $\mathcal{F}$ of sets of size $n$ has an infinite subset $\mathcal{G}\subseteq\mathcal{F}$ with a choice function.

	\item $\operatorname{LOC}_n^-$ states that every infinite, linearly orderable family $\mathcal{F}$ of sets of size $n$ has an infinite subset $\mathcal{G}\subseteq\mathcal{F}$ with a choice function.

	\item $\operatorname{WOC}_n^-$ states that every infinite, well-orderable family $\mathcal{F}$ of sets of size $n$ has an infinite subset $\mathcal{G}\subseteq\mathcal{F}$ with a choice function.

	\item $\operatorname{RC}_n$ states that every infinite set $X$ has an infinite subset $Y\subseteq X$ such that $[Y]^n=\{ z\subseteq Y\mid \lvert z\rvert=n\}$ has a choice function.

	\item Let $\mathcal{F}$ be an infinite family of $n$-element sets. A Kinna-Wagner selection function of $\mathcal{F}$ is a function $f$ with $\operatorname{dom}(f)=\mathcal{F}$ such that for all $p\in\mathcal{F}$, $\emptyset\not =f(p)\subsetneq p$.

	\item $\operatorname{KW}_n^-$ states that every infinite family $\mathcal{F}$ of sets of size $n$ has an infinite subset $\mathcal{G}\subseteq\mathcal{F}$ with a Kinna-Wagner selection function.

	\item $\operatorname{LOKW}_n^-$ states that every infinite, linearly orderable family $\mathcal{F}$ of sets of size $n$ has an infinite subset $\mathcal{G}\subseteq\mathcal{F}$ with a Kinna-Wagner selection function.
\end{enumerate}
\end{defi}

In 1999, Montenegro proved in \cite{Montenegro99} that $\operatorname{RC}_n\Rightarrow \operatorname{C}_n^-$ for all $n\in \{2,3,4\}$. It is still unknown whether this implication holds for any $n>5$. In 2017, Halbeisen and Tachtsis found some interesting results concerning the implications $\operatorname{RC}_m\Rightarrow \operatorname{C}_n^-$ and $\operatorname{RC}_m\Rightarrow \operatorname{RC}_n$ for $m,n\in\omega\setminus \{0,1\}$ (see \cite{Halbeisen17}). Among other results they were able to prove the following statements:
\begin{enumerate}
	\item[($\alpha$)] If $m,n\in\omega\setminus \{0,1\}$ are such that there is a prime $p$ with $p\nmid m$ and $p\mid n$, then 
$$
\operatorname{RC}_m\not\Rightarrow \operatorname{RC}_n\text{ and }\operatorname{RC}_m\not \Rightarrow \operatorname{C}_n^-
$$
in $\operatorname{ZF}$.
	\item[($\beta$)] $\operatorname{RC}_5$ implies none of $\operatorname{LOC}_2^-$and $\operatorname{LOC}_3^-$. 
	\item[($\gamma$)] For every $n\in\omega\setminus \{0,1\}$ we have that $\operatorname{C}_n^-\Rightarrow\operatorname{LOC}_n^-\Rightarrow\operatorname{WOC}_n^-$ but none of these implications is reversable in $\operatorname{ZF}$.
	\item[($\delta$)] For every $n\in\omega\setminus \{0,1\}$ the implication $\operatorname{RC}_{2n}\Rightarrow\operatorname{LOKW}_n^-$ holds. In particular we have that $\operatorname{RC}_6\Rightarrow \operatorname{LOC}_3^-$ in $\operatorname{ZF}$.
\end{enumerate}

In this paper we will give a full characterization of when $\operatorname{RC}_n\Rightarrow \operatorname{WOC}_m^-$ holds for $n,m\in\omega\setminus \{0,1\}$. To be more precise, we will show the following result:

\begin{thm}
\label{maintheorem}
Let $m,n\in\omega\setminus 2$. Then $\operatorname{RC}_m$ implies $\operatorname{WOC}_n^-$ if an only if the following condition holds: For all prime numbers $p_0,\dots, p_{k-1}$, $k\in \omega$, such that there are $a_0,\dots,a_{k-1}\in\omega$ with
$$
n=\sum_{i< k}a_ip_i,
$$
we can find $b_0,\dots, b_{k-1}\in\omega$ with 
$$
m=\sum_{i< k} b_ip_i.
$$
\end{thm}
In order to prove the independence result, we use permutation models. See \cite{Halbeisen12} for basics about permutation models and $\operatorname{ZFA}$ ($\operatorname{ZF}$ with the Axiom of Extensionality modified to allow atoms). With Pincus' transfer theorems (see \cite{Pincus72}), we will be able to transfer the results obtained in $\operatorname{ZFA}$ to $\operatorname{ZF}$.\newline

Theorem \ref{maintheorem} gives us the following three special cases:
\begin{enumerate}
	\item For all $n\in\omega$ we have that $\operatorname{RC}_n\Rightarrow\operatorname{WOC}_n^-$. (See Corollary \ref{Cor6.3}.)
	\item Let $p$ be a prime number, $m\in\omega\setminus \{0\}$ and $n\in\omega\setminus \{0,1\}.$ Then 
$$
\operatorname{RC}_{p^m}\Rightarrow \operatorname{WOC}_n^-
$$
if and only if $n\mid p^m$ or $p=2$, $m=1$ and $n=4$. (See Corollary \ref{Cor8.4}.)
	\item If $\operatorname{RC}_m\not \Rightarrow\operatorname{WOC}_n^-$, we also have that $\operatorname{RC}_m\not\Rightarrow \operatorname{RC}_n^-$  and $\operatorname{RC}_m\not\Rightarrow \operatorname{C}_n^-$ (see Corollary \ref{spezialfall}). This generalizes Halbeisen's and Tachtsis' result ($\alpha$).
\end{enumerate}

In the second part of this paper, we will give some insights into the question what happens when we weaken the assumption that our family of $n$-element sets is well-ordered. We will prove that
$\operatorname{RC}_n\Rightarrow\operatorname{LOC}_n^-$ for both $n\in \{5,7\}$ and that $\operatorname{RC}_6\Rightarrow \operatorname{C}_n^-$ for both $n\in\{3,9\}$.

\section{Why $\operatorname{RC}_6$ implies $\operatorname{C}_3^-$}

In this section we will closely follow the proof of $\operatorname{RC}_4\Rightarrow \operatorname{C}_4^-$ in \cite{Montenegro99}.

\begin{prop}
\label{case33}
$\operatorname{RC}_6$ implies $\operatorname{C}_3^-$.
\end{prop}

\begin{proof}
Let $\mathcal{F}$ be an infinite family of pairwise disjoint sets of size $3$. We apply $\operatorname{RC}_6$ to the set $\bigcup\mathcal{F}$. This gives us an infinite subset $Y\subseteq\bigcup\mathcal{F}$ with a choice function on $[Y]^6$. For every $i\in\{1,2,3\}$ we define
$$
\mathcal{G}_i:=\{ u\in\mathcal{F}\mid \lvert u\cap Y\rvert=i\}.
$$
Without loss of generality we can assume that $\mathcal{G}:=\mathcal{G}_3$ is infinite. So there is a choice function
$$
f:\left [\bigcup\mathcal{G}\right ]^6\to\bigcup\mathcal{G}.
$$
We define a directed graph on $\mathcal{G}$ by putting a directed edge from $v$ to $u$ if and only if $f(u\cup v)\in u$. With this graph we do the same construction as in \cite{Montenegro99}. So there is an infinite subset $\mathcal{H}\subseteq\mathcal{G}$ which is partitioned into finite sets $(A_n)_{n\in\omega}$ such that for every $n\in\omega$, all elements in $A_n$ have outdegree $n$. Moreover, for all $n<m$ the edges between $A_n$ and $A_m$ all point from $A_m$ to $A_n$ and $\lvert A_n\rvert$ is odd for all $n\in\omega$. We can assume that we are in one of the following two cases:\newline

\noindent \textit{Case 1: }There are infinitely many $n\in\omega$ with $3\nmid\lvert A_n\rvert$.\newline
\noindent In this case we follow the proof of the Claim in \cite[p.~60]{Montenegro99} and we are done.\newline

\noindent \textit{Case 2: }For all $n\in\omega$ we have that $3\mid \lvert A_n\rvert$.\newline
\noindent Let $p_0\in \mathcal{H}$ and let $n\in\omega$ be the unique natural number with $p_0\in A_n$. There is an $s\in\omega$ with $\lvert A_n\rvert =2s+1.$ We want to find the number of elements in $A_n$ that point to $p_0$. There are $\binom{\lvert A_n\rvert}{2}$ edges in $A_n$. Since the number of edges in $A_n$ that point to an element in $A_n$ is the same for every element of $A_n$, we have that the indegree of $p_0$ in $A_n$ is given by
$$
\operatorname{indegree}_{A_n}(p_0)=\frac{1}{\lvert A_n\rvert}\binom{\lvert A_n\rvert}{2}=\frac{1}{2}(\lvert A_n\rvert -1)=s.
$$
By assumption we have that $3\mid \lvert A_n\rvert=2s+1.$ Therefore, $3\nmid s$. Assume that $p_0=\{ x_0,x_1, x_2\}$. For every $i\leq 2$ we define
$$
A_n^{x_i}\coloneqq\{ v\in A_n\mid f(v\cup p_0)=x_i\}.
$$
Since $3\nmid (\lvert A_n^{x_0}\rvert+\lvert A_n^{x_1}\rvert+\lvert A_n^{x_2}\rvert)=s$, we can choose an element from $p_0$.
\end{proof}
\section{Why $\operatorname{RC}_6$ implies $\operatorname{C}_9^-$}

\begin{lem}[Kummer]
\label{binom}
Let $n,k\in\omega\setminus\{ 0\}$. The exponent of a prime $p$ in the prime factorization of $\begin{pmatrix}n\\k\end{pmatrix}$ is equal to the number of $j\in\omega\setminus\{ 0\}$ such that the fractional part of $\frac{k}{p^j}$ is greater than the fractional part of $\frac{n}{p^j}$.
\end{lem}

\begin{lem}
\label{lem3.2}
Let $\mathcal{F}$ be an infinite family of pairwise disjoint $4$-element sets and let 
$$
f:\left [ \bigcup \mathcal{F}\right]^6\to\bigcup\mathcal{F}
$$
be a choice function. Then there is a function $h$ with $h(p\cup q)\in p\cup q$ for all $p\not = q$ in $\mathcal{F}$. 
\end{lem}

\begin{proof}
Let $p\not = q$ be elements of $\mathcal{F}$. We will  show that we can choose exactly one element from $p\cup q$. There are 
$$
\begin{pmatrix}8\\6\end{pmatrix}=28
$$
$6$-element subsets of $p\cup q$. From each of these subsets we can choose one point with the choice function $f$. Let $A$ be the set of all elements in $p\cup q$ which are chosen the most times. Note that $1\leq\lvert A\rvert\leq 7$, because $8$ does not divide $28$. 

\renewcommand\labelitemi{{\boldmath$\cdot$}}
\begin{itemize}
	\item If $\lvert A\rvert=1$ we are done.
	\item If $\lvert A\rvert=2$, choose $f((p\cup q)\setminus A)$.
	\item If $\lvert A\rvert=3$ and $A\subseteq p$ or $A\subseteq q$ we are done because we can choose the point in $p\setminus A$ or in $q\setminus A$. Otherwise, $\lvert p\cap A\rvert=1$ or $\lvert q\cap A\rvert=1$ and we are also done.
	\item If $\lvert A\rvert\in\{5,6,7\}$, replace $A$ by $(p\cup q)\setminus A$. So we are in one of the cases above.
	\item If $\lvert A\rvert=4$, the set $[ (p\cup q)\setminus A]^2$ contains $\begin{pmatrix}4\\ 2\end{pmatrix}=6$ elements. For each $B\in [ (p\cup q)\setminus A]^2$ choose $f(A\cup B)$. Let $C_0$ and $C_1$ be the sets of all elements in $p\cup q$ which are chosen the most and the least often. Note that either $C_0$ or $C_1$ does not contain $4$ elements. By the cases above we are done.
\end{itemize}
So there is a choice function
$$
h:\{ p\cup q\mid p,q\in\mathcal{F}\}\to \bigcup \mathcal{F}.
$$
\end{proof}

\begin{lem}
\label{lem:2}
Let $\{ A_n\mid n\in\mathbb{N}\}$ be a countable family of pairwise disjoint non-empty finite sets of pairwise disjoint sets of size 2. Moreover, we assume that $\mathcal{F}\coloneqq\bigcup_{n\in\omega} A_n$ is an infinite family of $2$-element sets with a choice function 
$$
f:\left [\bigcup\mathcal{F}\right]^6\to\bigcup\mathcal{F}.
$$
Then there is an infinite subfamily $\mathcal{G}\subseteq\mathcal{F}$ with a choice function.
\end{lem}

\begin{proof}
We can assume that we are in one of the following four cases:\newline

\noindent \textit{Case 1: }For all $n\in\omega$ we have that $2\nmid\lvert A_n\rvert$.\newline
\noindent Let $k\in\omega$. Then there are natural numbers $l_0,l_1$ and $l_2$ such that
$$
\lvert A_{3k}\rvert=2l_0+1,~\lvert A_{3k+1}\rvert= 2l_1+1~\text{ and }~\lvert A_{3k+2}\rvert=2l_2+1.
$$
For every $a\in A_{3k}\cup A_{3k+1}\cup A_{3k+2}$ define
$$
\# a\coloneqq\lvert\{ (a_0,a_1,a_2)\in A_{3k}\times A_{3k+1}\times A_{3k+2}\mid f(a_0\cup a_1\cup a_2)\in a\}\rvert.
$$
If $\# a$ is odd, we can choose an element from $a$, for example the element in $a$ we choose more often than the other. Since 
$$
2\nmid\prod_{i\leq 2} (2l_i+1)~\text{ and }~\smashoperator[r]{\sum_{a\in A_{3k}\cup A_{3k+1}\cup A_{3k+2}}} ~\#a~~~~=\prod_{i\leq 2} (2l_i+1),
$$
we have that for every $k\in\omega$ there is at least one $a\in A_{3k}\cup A_{3k+1}\cup A_{3k+2}$ such that $\# a$ is odd. So we can find a choice function on the infinite set
$$
\mathcal{G}\coloneqq\{ a\in\mathcal{F}\mid \# a \text{ is odd}\}.
$$

\noindent \textit{Case 2: }For all $n\in\omega$ we have that $\lvert A_n\rvert=2$.\newline
\noindent For every $k\in\omega$ there are two distinct 3-element subsets $B_0$ and $B_1$ of $A_{2k}\cup A_{2k+1}$ such that $A_{2k+1}\subseteq B_0$ and $A_{2k+1}\subseteq B_1$. For every $a\in A_{2k}\cup A_{2k+1}$ we define
$$
\# a\coloneqq\left\lvert\left\{ i\in \{0,1\}\mid f\left(\bigcup B_i\right)\in a\right\}\right\rvert.
$$
Note that if $\# a$ is odd we can choose an element from $a$. So if there are infinitely many $a\in\mathcal{F}$ such that $\# a$ is odd, we are done. Otherwise, there is an infinite subset $I\subseteq\omega$ such that for all $k\in I$ there is a unique $a_k\in A_{2k}\cup A_{2k+1}$ with $\# a_k=2$. Then we are in the first case for the family $\{ \{ a_k\}\mid k\in I\}.$\newline

\noindent \textit{Case 3: }For all $n\in\omega$ we have that $\lvert A_n\rvert\geq 3$, $4\nmid \lvert A_n\rvert$ and $2\mid \lvert A_n\rvert$.\newline
\noindent Let $n\in\omega$. By Lemma \ref{binom}, $\begin{pmatrix}\lvert A_n\rvert\\ 2\end{pmatrix}$ is odd.  For every $k\in\omega$ we look at the 4-element subsets of $A_{2k}\cup A_{2k+1}$ with two elements in $A_{2k}$ and two elements in $A_{2k+1}$. Note that the number of such subsets is odd. Let $h$ be the choice function we found in Lemma \ref{lem3.2}. Then for every $k\in\omega$ there is at least one $a\in A_{2k}\cup A_{2k+1}$ such that
$$
\# a\coloneqq\lvert\{ (\{ a_0,a_1\}, \{ b_0, b_1\})\in [A_{2k}]^2\times [A_{2k+1}]^2\mid h(a_0\cup a_1\cup b_0\cup b_1)\in a\}\rvert
$$
is odd. So again we found a choice function on the infinite set
$$
\mathcal{G}\coloneqq\{ a\in\mathcal{F}\mid \# a\text{ is odd}\}.
$$

\noindent \textit{Case 4: }For all $n\in\omega$ we have that $\lvert A_n\rvert\geq 3$ and $4\mid \lvert A_n\rvert$.\newline
\noindent Let $n\in\omega$. Then there is a $k\in\omega$ with $\lvert A_n\rvert= 4k$. We have that
\begin{equation}
\label{eq:1}
2\lvert A_n\rvert\nmid \binom{\lvert A_n\rvert}{3},
\end{equation}
since otherwise we would have that 
$$
\frac{\lvert A_n\rvert(\lvert A_n\rvert-1)(\lvert A_n\rvert-2)}{2\cdot\lvert A_n\rvert\cdot 2\cdot 3}=\frac{ 2(4k^2-3k)+1}{2\cdot 3}\in\omega.
$$
But this is not the case since the numerator is not divisible by 2. We define
$$
\# a\coloneqq\lvert\{ \{ a_0,a_1,a_2\}\in [A_n]^3\mid f(a_0\cup a_1\cup a_2)\in a\}\rvert
$$
and for all $y \in \bigcup A_n$ let
$$
\operatorname{no}(y)\coloneqq\lvert\{ \{ a_0,a_1,a_2\}\in [A_n]^3\mid f(a_0\cup a_1\cup a_2)=y\}\rvert.
$$
Note that by $(\ref{eq:1})$ 
$$
\left\lvert\left\{ y\in \bigcup A_n\mid \operatorname{no}(y)=\max\left\{ \operatorname{no}(z)\mid z\in \bigcup A_n\right\}\right\}\right\rvert<2\lvert A_n\rvert.
$$
If there is an $a=\{ a_0,a_1\}\in A_n$ with 
$$
\operatorname{no}(a_0)\not = \operatorname{no}(a_1)
$$
choose the element with lower $\operatorname{no}$. Otherwise we have that
$$
B_n\coloneqq\{ a\in A_n\mid \# a=\max\{ \# b\mid b\in A_n\}\}\subsetneq A_n.
$$
Repeat the procedure  with $A_n\coloneqq B_n$ until either $4\nmid \lvert A_n\rvert$ or there is an $a=\{ a_0,a_1\}\in A_n$ with 
$$
\operatorname{no}(a_0)\not = \operatorname{no}(a_1).
$$
Note that we have to repeat the procedure at most $\lvert A_n\rvert$ times. In the end we either found a choice function on an infinite subset of $\mathcal{F}$ or we reduced Case 4 to one of the other cases.
\end{proof}

\begin{cor}
\label{cor:3.4}
Let $\mathcal{F}$ be an infinite family of pairwise disjoint $4$-element sets and let 
$$
f:\left [ \bigcup \mathcal{F}\right]^6\to\bigcup\mathcal{F}
$$
be a choice function. Then there is an infinite subset $\mathcal{G}\subseteq\mathcal{F}$ with a choice function on $\mathcal{G}$.
\end{cor}

\begin{proof}
Let $h$ be the choice function we found in Lemma \ref{lem3.2}. We can define a complete, directed graph on $\mathcal{F}$ by putting an edge from $p$ to $q$ if and only if $g(p\cup q)\in q$. With this graph we can do the same construction as in \cite{Montenegro99}. So we can find an infinite subset $\mathcal{G}\subseteq\mathcal{F}$ such that we can choose exactly $1$ or $2$ elements from each $G\in\mathcal{G}$. So either we found a choice function on an infinite subset of $\mathcal {G}$ or we can find an infinite family of $2$-element sets $\mathcal{H}$. Then we apply Lemma \ref{lem:2} to $\mathcal{H}$ and we are done.
\end{proof}

\begin{lem}
\label{lemma225}
Let $\mathcal{F}$ be an infinite family of $10$-element sets. Assume that each $F\in\mathcal{F}$ is a disjoint union of five $2$-element sets $F_i$, $0\leq i\leq 4$. Moreover, let 
$$
f:\left[\bigcup \mathcal{F} \right ]^6\to \bigcup \mathcal{F}
$$
be a choice function. Then there is an infinite subset $\mathcal{G}\subseteq\mathcal{F}$ with a Kinna-Wagner selection function.
\end{lem}

\begin{proof}
For all $4$-element sets $A\subseteq\bigcup\mathcal{F}$, we define the degree of $A$ by
$$
\operatorname{deg}(A):=\lvert\{ F_i\mid F\in\mathcal{F}\land i\leq 4\land f(A\cup F_i)\in F_i\}\rvert.
$$
If there is an $A_0\in\left [\bigcup \mathcal{F}\right ]^4$ with infinite degree we are done. Because then the set 
$$
\mathcal{G}:=\{ F\in\mathcal{F}\mid \exists i\leq 4~(f(A_0\cup F_i)\in F_i)\}
$$ 
is infinite and from every $G\in\mathcal{G}$ we can choose the set 
$$
\{ f(A_0\cup G_i)\mid i\leq 4\}\cap G\subsetneq G.
$$
So we can assume that each $A\in\left [\bigcup \mathcal{F}\right ]^4$ has finite degree. Define $
\mathcal{F}^2=\{ F_i\mid i\leq 4\land F\in\mathcal{F}\}$ and for all $F\in\mathcal{F}$ let $F^2:=\{ F_i\mid i\leq 4\}$.

\noindent \textit{Case 1: }There is an $n\in\omega$ such that for infinitely many $F\in\mathcal{F}$ there are $A,B\in F^2$ with $\operatorname{deg}(A\cup B)=n$.\newline
Let $\mathcal{G}:=\{ F\in\mathcal{F}\mid \exists A,B\in F^2(\operatorname{deg}(A\cup B)=n)\}$. By assumption this is an infinite set. Choose an $(n+3)$-element set $\{X_i\mid i\leq n+2\}\subseteq\mathcal{F}^2$. For all $G\in\mathcal{G}$ and all $A,B\in G^2$ with $\operatorname{deg}(A\cup B)=n$ put an edge pointing from $A$ to $B$ if and only if
$$
f(A\cup B\cup X_{i_0})\in B,
$$
where
$$
i_0:=\min \{ i\leq n+2\mid f(A\cup B\cup X_i)\notin X_i\}.
$$
So we defined a directed graph with at least one edge in each $G^2$ with $G\in\mathcal{G}$. If for infinitely many $G\in\mathcal{G}$ not all elements of $G^2$ have the same outdegree, we are done. So we either have a cycle on infinitely many $G^2$ or we have a complete graph in which every node has outdegree $2$. In the first case we can choose a point in each $A\cup B$, where $A,B\in G^2$ are neighbours. So we can choose $5$ elements in each $G\in\mathcal{G}$. In the second case we can choose $5$ edges as follows: For each node $A\in G^2$, let $B,C\in G^2$ be the two successors of $A$ in the graph. Choose the edge which connects $B$ and $C$ (see Figure \ref{figureone}). We can again choose at most $5$ elements from $G$.\newline

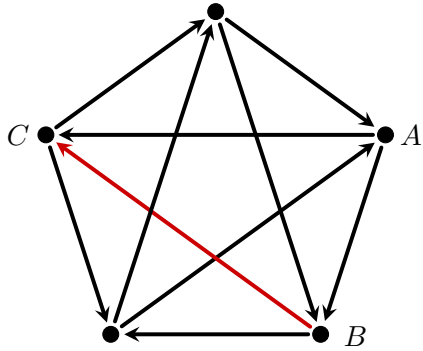
\begin{figure}[h]
\definecolor{ccqqqq}{rgb}{0.8,0.,0.}
\begin{center}
\begin{tikzpicture}[line cap=round,line join=round,>=stealth,x=0.8cm,y=0.8cm]
\clip(-0.3,-0.6) rectangle (6.8,5.9);
\draw [->,line width=1.5pt] (0.5919476561488441,3.441295204757757) -- (3.0683207790886495,5.240485593782973);
\draw [->,line width=1.5pt] (3.424912450030659,5.234361379399854) -- (5.880097178517175,3.450565260045104);
\draw [->,line width=1.5pt] (5.995642740353123,3.1066407512097185) -- (5.058141959772239,0.22131003272698058);
\draw [->,line width=1.5pt] (4.7794170519382915,0.) -- (1.7218819341055425,0.);
\draw [->,line width=1.5pt] (0.4881724535946267,3.1096838945740877) -- (1.4352500512878887,0.19487876372604607);
\draw [->,line width=1.5pt] (1.6811746009135269,0.1326700083731238) -- (5.886185992737765,3.1877896152804457);
\draw [->,line width=1.5pt,color=ccqqqq] (4.812746623335892,0.12604593626891925) -- (0.5910585515265581,3.1932818604113504);
\draw [->,line width=1.5pt] (5.855909679753793,3.3169655464245844) -- (0.6292192330563697,3.3169655464245853);
\draw [->,line width=1.5pt] (3.3114144083452732,5.154564596230232) -- (4.9199901131597805,0.2038776312225412);
\draw [->,line width=1.5pt] (1.5683244577723467,0.21468214633053861) -- (3.175050766020351,5.15967725397179);
\draw (6.14,3.65) node[anchor=north west] {$A$};
\draw (5.2,0.3) node[anchor=north west] {$B$};
\draw (-0.4,3.65) node[anchor=north west] {$C$};
\begin{scriptsize}
\draw [fill=black] (1.49857,0.) circle (3pt);
\draw [fill=black] (4.986233971165293,0.) circle (3pt);
\draw [fill=black] (6.063981408924586,3.3169655464245844) circle (3pt);
\draw [fill=black] (3.242401985582647,5.366962993627345) circle (3pt);
\draw [fill=black] (0.420822562240708,3.3169655464245853) circle (3pt);
\end{scriptsize}
\end{tikzpicture}
\caption{How to choose the edges}
\label{figureone}
\end{center}
\end{figure}

\noindent\textit{Case 2: }For all $n\in\omega$ there are only finitely many $F\in\mathcal{F}$ such that there are $A,B\in F^2$ with $\operatorname{deg}(A\cup B)=n$.\newline
Let $A_{-1}:=\emptyset$ and for every $n\in\omega$ define 
$$
A_n:=\{A\in\mathcal{F}\mid \exists B\in\mathcal{F}(\operatorname{deg}(A\cup B)=n)\}\setminus A_{n-1}.
$$
Note that these sets are pairwise disjoint families of $2$-element sets. So we can apply Lemma \ref{lem:2} and we are done.
\end{proof}

\begin{prop}
$\operatorname{RC}_6$ implies $\operatorname{C}_9^-$.
\end{prop}

\begin{proof}
Let $\mathcal{F}$ be an infinite family of pairwise disjoint sets of size $9$. Since $\operatorname{RC}_9$ holds, there is an infinite set $Y\subseteq\bigcup\mathcal{F}$ with a choice function 
$$
f:[Y]^6\to Y.
$$
For all $0\leq i\leq 9$ let 
$$
\mathcal{G}_i:=\{F\cap Y\mid F\in\mathcal{F}\land\lvert F\cap Y\rvert=i\}.
$$
There is a $1\leq i\leq 9\setminus \{0\}$ such that $\mathcal{G}_i$ is an infinite set.\newline

\noindent\textit{Case 1: }$\mathcal{G}_1$ or $\mathcal{G}_8$ is infinite.\newline
In the case $\mathcal{G}_8$ is infinite, we look at the complements.\newline

\noindent\textit{Case 2: }$\mathcal{G}_3$ or $\mathcal{G}_6$ is infinite.\newline
Use Proposition \ref{case33}.\newline

\noindent\textit{Case 3: }$\mathcal{G}_4$  is infinite.\newline
Use Corollary \ref{cor:3.4}.\newline

\noindent\textit{Case 4: }$\mathcal{G}_5$ is infinite.\newline
Apply $\operatorname{RC}_6$ to the complements. Then we are either in one of the preceding cases  or the complements are partitioned into two sets of size two. We look at the $10$ edges between the first $5$ elements and the second two elements and use Lemma \ref{lemma225}.\newline

\noindent\textit{Case 5: }$\mathcal{G}_7$ is infinite.\newline
For all $G\in\mathcal{G}_7$ let $\overline{G}$ be the complement of $G$ in the sense that for the $F\in\mathcal{F}$ with $G\subseteq F$ we have that 
$$
\overline{G}:=F\setminus G.
$$
Note that $\lvert \overline{G}\rvert=2$. Let 
$$
\mathcal{E}:=\{ \{x,y\}\mid \exists G\in\mathcal{G}_7 (x\in G\text{ and }y\in \overline{G})\}.
$$
Apply $\operatorname{RC}_6$ to $\mathcal{E}$. Without loss of generality we can assume that we find a choice function 
$$
g:[\mathcal{E}]^6\to\mathcal{E},
$$
because otherwise we are in one of the preceding cases. So for every $G\in\mathcal{G}_7$ there are $14$ edges between $G$ and $\overline{G}$. So there are 
$$
\begin{pmatrix}14\\6\end{pmatrix}=3\cdot 7\cdot 11\cdot 13
$$
$6$-element subsets. From each of them $g$ chooses one element. Since $\begin{pmatrix}14\\6\end{pmatrix}$ is not divisible by $14$, we can choose less than $14$ edges and we are in one of the preceding cases.\newline

\noindent\textit{Case 6: }$\mathcal{G}_9$ is infinite.\newline
With the choice function $f$ we can choose an element from each $6$-element subset of a $G\in\mathcal{G}_9$. There are $\begin{pmatrix}9\\6\end{pmatrix}$ subsets of size $6$. Since $9\nmid \begin{pmatrix}9\\6\end{pmatrix}$ we can reduce this case to one of the cases above.\newline

\noindent\textit{Case 7: }$\mathcal{G}_2$ is infinite.\newline
We iteratively apply $\operatorname{RC}_6$ to the complements. So we can reduce this case to one of the cases above.
\end{proof}

\section{Why $\operatorname{RC}_5$ implies $\operatorname{LOC}_5^-$}

In this section we will prove that $\operatorname{RC}_5$ implies $\operatorname{LOC}_5^-$. The beginning of the proof will be as usual: Let $\mathcal{F}$ be an infinite, linearly orderable family of $5$-element sets. We apply $\operatorname{RC}_5$ to $\bigcup\mathcal{F}$. This will give us an infinite subfamily $\mathcal{G}\subseteq\mathcal{F}$ such that each $p\in \mathcal{G}$ is partitioned into two parts, one of size $2$ and one of size $3$. So if we could show that $\operatorname{RC}_5$ implies $\operatorname{LOC}_2^-$ or $\operatorname{LOC}_3^-$, the proof would be finished. However, Halbeisen's and Tachtsis' result ($\beta$) shows that this idea will not lead to success. That's why we will work with the set of all edges between the two parts.

\begin{thm}
\label{Thm3.1}
We have that $\operatorname{RC}_5$ implies $\operatorname{LOC}_5^-$.
\end{thm}

\begin{proof}
Let $\mathcal{F}$ be an infinite, linearly orderable collection of pairwise disjoint sets of size 5. We fix a linear order on $\mathcal{F}$ and apply $\operatorname{RC}_5$ on the set $X\coloneqq\bigcup \mathcal{F}$ to find an infinite subset $Y\subseteq X$ with a choice function $f:[Y]^5\to Y$. For every $i\leq 5$ we define
$$
\mathcal{F}_i\coloneqq\{ p\in\mathcal{F}\mid \lvert p\cap Y\rvert=i\}.
$$
Without loss of generality we can assume that by the $\mathcal{F}_i$, the elements $p$ of an infinite subfamily $\mathcal{G}\subseteq\mathcal{F}$ are partitioned into a set with two elements and a set with three elements, namely $p=\{ a_p,b_p,c_p\}\cup \{ x_p,y_p\}$. \newline

\noindent Now we look at the set $Z$ of all non-directed edges between a point in $\{a_p,b_p,c_p\}$ and one in $\{x_p,y_p\}$. For every $p\in\mathcal{G}$ let $p^*$ be the set of all edges in $Z$ belonging to $p$ and for each subset $\mathcal{H}\subseteq\mathcal{F}$ we define $\mathcal{H}^*\coloneqq\{ p^*\mid p\in\mathcal{H}\}$. \newline

\textit{Claim 1:} Assume that there is an infinite subset $\mathcal{H}\subseteq \mathcal{G}$ such that we can choose between $1$ and $5$ elements from each $p^*\in \mathcal{H}^*$. Then there is a choice function
$$
h:\mathcal{H}\to\bigcup\mathcal{H}.
$$
\begin{proof}[Proof of Claim 1] Let $p\in\mathcal{H}$ and assume that we can choose $k\in\{ 1,2,3,4,5\}$ elements from $p^*$. We look at $p$ as a graph with $k$ edges. If $2\nmid k$, $x_p$ and $y_p$ do not have the same degree and we can choose the element with lower degree. Otherwise we have that $3\nmid k$ and we can choose an element from $\{ a_p,b_p,c_p\}$.
\renewcommand{\qedsymbol}{{\large$\boldsymbol{\dashv}_{\text{Claim 1}}$}}
\end{proof}

Now we apply $\operatorname{RC}_5$ on the set $Z$. Then there is an infinite subset $Q\subseteq Z$ with a choice function $g:[Q]^5\to Q$. By Claim 1 we can without loss of generality assume that $p^*\subseteq Q$ for every $p\in \mathcal{H}$.\newline

\noindent We can partition each $p^*\in\mathcal{H}^*$ as follows into two sets $\gamma_0^p$ and $\gamma_1^p$ of size three: 
$$
\gamma_0^p\coloneqq\{\{ a_p, x_p\},\{b_p,x_p\},\{ c_p,x_p\}\}~\text{ and }~\gamma_1^p\coloneqq\{\{ a_p, y_p\},\{b_p,y_p\},\{ c_p,y_p\}\}.
$$
Analogously we can partition $p^*$ into three sets $\beta_0, \beta_1, \beta_2$ of size two as follows:
$$
\beta_0^p\coloneqq\{\{a_p,x_p\},\{a_p,y_p\}\}, ~\beta_1^p\coloneqq\{\{b_p,x_p\},\{b_p,y_p\}\}~\text{ and }~\beta_2^p\coloneqq\{\{c_p,x_p\},\{c_p,y_p\}\}.
$$

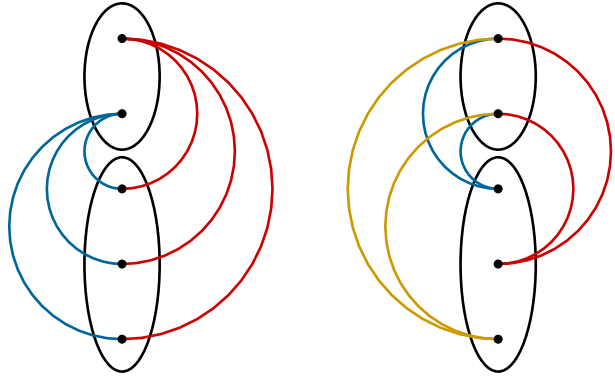
\begin{figure}[h]
\definecolor{cczzqq}{rgb}{0.8,0.6,0.}
\definecolor{qqwwzz}{rgb}{0.,0.4,0.6}
\definecolor{ccqqqq}{rgb}{0.8,0.,0.}
\begin{center}
\begin{tikzpicture}[line cap=round,line join=round,>=triangle 45,x=1.0cm,y=1.0cm]
\clip(0.3,0.3) rectangle (8.7,5.6);
\draw [rotate around={90.:(2.,4.495)},line width=1pt] (2.,4.495) ellipse (0.9732582382642967cm and 0.5000065983057722cm);
\draw [rotate around={90.:(2.,1.995)},line width=1pt] (2.,1.995) ellipse (1.4255623716554962cm and 0.500003075470584cm);
\draw [rotate around={90.:(7.,4.495)},line width=1pt] (7.,4.495) ellipse (0.9732582382642712cm and 0.5000065983057591cm);
\draw [rotate around={90.:(7.,1.995)},line width=1pt] (7.,1.995) ellipse (1.4255623716554962cm and 0.500003075470584cm);
\draw [shift={(2.,4.)},color=ccqqqq,line width=1pt]  plot[domain=-1.5707963267948966:1.5707963267948966,variable=\t]({1.*1.*cos(\t r)+0.*1.*sin(\t r)},{0.*1.*cos(\t r)+1.*1.*sin(\t r)});
\draw [shift={(2.,3.5)},color=ccqqqq,line width=1pt]  plot[domain=-1.5707963267948966:1.5707963267948966,variable=\t]({1.*1.5*cos(\t r)+0.*1.5*sin(\t r)},{0.*1.5*cos(\t r)+1.*1.5*sin(\t r)});
\draw [shift={(2.,3.)},color=ccqqqq,line width=1pt]  plot[domain=-1.5707963267948966:1.5707963267948966,variable=\t]({1.*2.*cos(\t r)+0.*2.*sin(\t r)},{0.*2.*cos(\t r)+1.*2.*sin(\t r)});
\draw [shift={(2.,3.5)},color=qqwwzz,line width=1pt]  plot[domain=1.5707963267948966:4.71238898038469,variable=\t]({1.*0.5*cos(\t r)+0.*0.5*sin(\t r)},{0.*0.5*cos(\t r)+1.*0.5*sin(\t r)});
\draw [shift={(2.,3.)},color=qqwwzz,line width=1pt]  plot[domain=1.5707963267948966:4.71238898038469,variable=\t]({1.*1.*cos(\t r)+0.*1.*sin(\t r)},{0.*1.*cos(\t r)+1.*1.*sin(\t r)});
\draw [shift={(2.,2.5)},color=qqwwzz,line width=1pt]  plot[domain=1.5707963267948966:4.71238898038469,variable=\t]({1.*1.5*cos(\t r)+0.*1.5*sin(\t r)},{0.*1.5*cos(\t r)+1.*1.5*sin(\t r)});
\draw [shift={(7.,4.)},color=qqwwzz,line width=1pt]  plot[domain=1.5707963267948966:4.71238898038469,variable=\t]({1.*1.*cos(\t r)+0.*1.*sin(\t r)},{0.*1.*cos(\t r)+1.*1.*sin(\t r)});
\draw [shift={(7.,3.5)},color=qqwwzz,line width=1pt]  plot[domain=1.5707963267948966:4.71238898038469,variable=\t]({1.*0.5*cos(\t r)+0.*0.5*sin(\t r)},{0.*0.5*cos(\t r)+1.*0.5*sin(\t r)});
\draw [shift={(7.,2.5)},color=cczzqq,line width=1pt]  plot[domain=1.5707963267948966:4.71238898038469,variable=\t]({1.*1.5*cos(\t r)+0.*1.5*sin(\t r)},{0.*1.5*cos(\t r)+1.*1.5*sin(\t r)});
\draw [shift={(7.,3.)},color=cczzqq,line width=1pt]  plot[domain=1.5707963267948966:4.71238898038469,variable=\t]({1.*2.*cos(\t r)+0.*2.*sin(\t r)},{0.*2.*cos(\t r)+1.*2.*sin(\t r)});
\draw [shift={(7.,3.)},color=ccqqqq,line width=1pt]  plot[domain=-1.5707963267948966:1.5707963267948966,variable=\t]({1.*1.*cos(\t r)+0.*1.*sin(\t r)},{0.*1.*cos(\t r)+1.*1.*sin(\t r)});
\draw [shift={(7.,3.5)},color=ccqqqq,line width=1pt]  plot[domain=-1.5707963267948966:1.5707963267948966,variable=\t]({1.*1.5*cos(\t r)+0.*1.5*sin(\t r)},{0.*1.5*cos(\t r)+1.*1.5*sin(\t r)});
\begin{scriptsize}
\draw [fill=black] (2.,5.) circle (1.5pt);
\draw [fill=black] (2.,4.) circle (1.5pt);
\draw [fill=black] (2.,3.) circle (1.5pt);
\draw [fill=black] (2.,2.) circle (1.5pt);
\draw [fill=black] (2.,1.) circle (1.5pt);
\draw [fill=black] (7.,5.) circle (1.5pt);
\draw [fill=black] (7.,4.) circle (1.5pt);
\draw [fill=black] (7.,3.) circle (1.5pt);
\draw [fill=black] (7.,2.) circle (1.5pt);
\draw [fill=black] (7.,1.) circle (1.5pt);
\end{scriptsize}
\end{tikzpicture}
\end{center}
\caption{The  partitions of a $p^*$ into {\color{ccqqqq}$\gamma_0^p$}, {\color{qqwwzz}$\gamma_1^p$} on the left and into {\color{ccqqqq}$\beta_0^p$}, {\color{qqwwzz}$\beta_1^p$} and {\color{cczzqq}$\beta_2^p$} on the right.}
\end{figure}

Let 
$$
\mathcal{H}_3^*\coloneqq\{\gamma_i^p\mid i\leq 1\land p\in\mathcal{H}\}
$$
be the sets of size three appearing in the partition of a $p^*\in\mathcal{H}^*$ and let 
$$
\mathcal{H}_2^*\coloneqq\{\beta_i^p\mid i\leq 2\land p\in\mathcal{H}\}
$$
be the family of sets of size two which appear in the partition of a $p^*\in\mathcal{H}^*$. If there is a $\gamma \in \mathcal{H}_3^*$ such that for infinitely many $\beta\in \mathcal{H}_2^*$
\begin{equation}
\label{eq:a}
g(\gamma\cup\beta)\in \beta,
\end{equation}
we are done by Claim 1. Otherwise, for every $\gamma \in \mathcal{H}_3^*$ there are only finitely many $\beta\in \mathcal{H}_2^*$ with (\ref{eq:a}) and we define
$$
\operatorname{deg}(\gamma)\coloneqq\lvert\{\beta\in\mathcal{H}_2^*\mid g(\gamma\cup\beta)\in \beta\}\rvert\in \omega.
$$
We are in one of the following two cases:\newline

\noindent \textit{Case 1:} There is an $n\in\omega$ such that $\operatorname{deg}(\gamma)=n$ for infinitely many $\gamma\in \mathcal{H}_3^*$.\newline
\noindent Let $\mathcal{I}_3^*\coloneqq\{\gamma\in\mathcal{H}_3^*\mid \operatorname{deg}(\gamma)=n\}$. Choose an $(n+4)$-element set $\{\beta_i\mid i\leq n+3\}\subseteq\mathcal{H}_2^*$. For every $\gamma\in \mathcal{I}_3^*$ we define
$$
j(\gamma)\coloneqq\min\{i\leq n+3\mid g(\gamma\cup \beta_i)\in\gamma\}.
$$
So from every $\gamma\in\mathcal{I}_3^*$ we choose the element
$$
g(\gamma\cup \beta_{j(\gamma)})\in\gamma
$$
and we are done by Claim 1.\newline

\noindent \textit{Case 2:} For each $n\in\omega$ there are only finitely many $\gamma\in\mathcal{H}_3$ with $\operatorname{deg}(\gamma)=n$.\newline
\noindent For every $n\in\omega$ we define
$$
A_n\coloneqq\{\gamma\in\mathcal{H}_3^*\mid\operatorname{deg}(\gamma)=n\}\text{ and } B_n\coloneqq\{\beta\in \mathcal{H}_2^*\mid\exists \gamma\in A_n\exists p^*\in\mathcal{H}^*(\gamma\subseteq p^*\land \beta\subseteq p^*)\}.
$$
If there are infinitely many $p\in \mathcal{H}$ such that $\gamma^p_0\in A_n$ and $\gamma_1^p\in A_m$ with $n\not =m$ we are done by Claim 1 since we can choose three edges from each of these infinitely many $p$'s. So we can assume that for every $p\in\mathcal{H}$ both, $\gamma_0^p$ and $\gamma_1^p$, have the same degree and we define
$$
C_n\coloneqq\{ p\in \mathcal{H}\mid \{\gamma_0^p,\gamma_1^p\}\subseteq A_n\}
$$
for every $n\in\omega$. Moreover, let
$$
\operatorname{out}(\beta)\coloneqq\left\{\gamma\in\bigcup_{m>n} A_m\mid g(\beta\cup\gamma)\in\gamma\right\}.
$$
for every $n\in\omega$ and every $\beta\in B_n$. If there is a $\beta\in\bigcup_{n\in\omega} B_n$ with $\lvert \operatorname{out}(\beta)\rvert=\infty$ we are done by Claim 1. So assume that $\lvert \operatorname{out}(\beta)\rvert\in\omega$ for all $\beta\in\bigcup_{n\in\omega}B_n$.\newline

\noindent \textit{Claim 2:} We can construct an infinite subset $\mathcal{J}\subseteq\mathcal{H}$ with a partition $\mathcal{J}=\bigcup_{n\in\omega} J_n$, where each $J_n$ is finite. Moreover, for all natural numbers $n>m$, all $p\in J_n$, all $q\in J_m$, all $i\in\{ 0,1\}$ and all $j\in\{ 0,1,2\}$
$$
g(\gamma_i^p\cup \beta_j^q)\in \beta_j^q.
$$
\begin{proof}[Proof of Claim 2] 
For each $n\in\omega$ let $R_n$ be the finite set of all $p\in\bigcup_{m>n}C_m$ such that there are a $q\in C_n$, an $i\in \{0,1\}$ and a $j\in \{0,1,2\}$ with
$$
g(\gamma_i^p\cup \beta_j^q)\in\gamma_i^p.
$$
Let $J_n:=C_n\setminus R_n$.
\renewcommand{\qedsymbol}{{\large$\boldsymbol{\dashv}_{\text{Claim 2}}$}}
\end{proof}

With the same construction we did in the proof of Claim 2 we can find an infinite subset $\mathcal{K}\subseteq\mathcal{J}$ with a partition $\mathcal{K}=\bigcup_{n\in\omega} K_n,$ where each $K_n$ is finite and non-empty. Moreover, we can assume that for all natural numbers $n>m$, all $p\in I_n$, all $q\in I_m$ and all $j\leq 2$
$$
g(\gamma_0^p\cup\beta_j^q)=g(\gamma_1^p\cup\beta_j^q)\in \beta_j^q.
$$
Note: Up to now we nowhere used the assumption that our infinite family $\mathcal{F}$ of sets of size five is linearly ordered. In the last step we will need this assumption.\newline

\noindent For each $n\in\omega$, let $p_n\in K_n$ be the smallest element in $K_n$ with respect to the linear order on $\mathcal{F}$. Note that such a smallest element exists since each $K_n$ is finite and non-empty. We define
\begin{align*}
h^*:\{ p_n^*\mid n\in\omega\}&\to\left [\bigcup_{n\in\omega} p_n^*\right ]^3\\
p_n^*&\mapsto \left\{ g\left(\gamma_0^{p_{n+1}}\cup \beta_j^{p_n}\right)\mid j\leq 2\right\}.
\end{align*}
By Claim 1 we are done.
\end{proof}


\section{Why $\operatorname{RC}_7$ implies $\operatorname{LOC}_7^-$}

\begin{lem}
\label{lemma1}
Let $\mathcal{F}$ be a linearly orderable family of pairwise disjoint $6$-element sets. We assume that we can partition each $p\in\mathcal{F}$ in a unique way into three $2$-element sets $\beta_0^p,\beta_1^p$ and $\beta_2^p$ and in a unique way into two $3$-element sets $\gamma_0^p,\gamma_1^p$. Moreover, we require that there is a choice function
$$
f:\left [\bigcup \mathcal{F}\right]^7\to\bigcup \mathcal{F}.
$$
Then there is an infinite subfamily $\mathcal{G}\subseteq\mathcal{F}$ with a Kinna-Wagner selection function.
\end{lem}

\begin{proof}
We define
$$
\mathcal{F}_3:=\{\gamma_i^p\mid i\in\{ 0,1\}\land p\in\mathcal{F}\}
$$
and 
$$
\mathcal{F}_4:=\{\beta_i^p\cup\beta_j^p\mid \{i,j\}\in [3]^2\land p\in\mathcal{F}\}.
$$
For every $\gamma\in\mathcal{F}_3$ let
$$
\operatorname{deg}(\gamma):=\lvert\{ \delta\in\mathcal{F}_4\mid\delta\cap \gamma=\emptyset\land f(\delta\cup\gamma)\in\delta\}\rvert.
$$
If there is a $\gamma\in\mathcal{F}_3$ with $\operatorname{deg}(\gamma)=\infty$ we are done because we can choose between one and three edges from infinitely many elements of $\mathcal{F}$. The rest of the proof is similar to the proof of Theorem \ref{Thm3.1}.
\end{proof}

\begin{lem}
\label{lemma2}
Let $\mathcal{F}$ be a linearly orderable family of pairwise disjoint $12$-element sets. We assume that we can partition each $p\in\mathcal{F}$ in a unique way into three $4$-element sets $\delta_0,\delta_1$ and $\delta_2$ and in a unique way into four $3$-element sets $\gamma_0,\gamma_1,\gamma_2$ and $\gamma_3$. Moreover, we require that there is a choice function
$$
f:\left[\bigcup\mathcal{F}\right]^7\to\bigcup \mathcal{F}.
$$
Then there is an infinite subset $\mathcal{G}\subseteq\mathcal{F}$ with a Kinna-Wagner selection function.
\end{lem}

\begin{proof}
The proof is similar to the proof of Theorem \ref{Thm3.1}.
\end{proof}

\begin{lem}
\label{lemma3}
Let $\mathcal{F}$ be a linearly orderable family of pairwise disjoint $10$-element sets. We assume that we can partition each $p\in\mathcal{F}$ in a unique way into two $5$-element sets $\epsilon_0$ and $\epsilon_1$ and in a unique way into five $2$-element sets $\beta_i,~i\leq 4$. Moreover, we require that there is a choice function
$$
f:\left[\bigcup\mathcal{F}\right]^7\to\bigcup \mathcal{F}.
$$
Then there is an infinite subset $\mathcal{G}\subseteq\mathcal{F}$ with a Kinna-Wagner selection function.
\end{lem}

\begin{proof}
The proof is similar to the proof of Theorem \ref{Thm3.1}.
\end{proof}

\begin{prop}
We have that $\operatorname{RC}_7\Rightarrow\operatorname{LOC}_7^-$.
\end{prop}

\begin{proof}
Let $\mathcal{F}$ be a linearly orderable, infinite family of sets of size $7$. We apply $\operatorname{RC}_7$ on the set $X:=\bigcup\mathcal{F}$ to find an infinite subset $Y\subseteq X$ with a choice function $f:[Y]^7\to Y$. For every $i\leq 7$ we define
$$
\mathcal{F}_i:=\{ p\in\mathcal{F}\mid \lvert p\cap Y\rvert=i\}.
$$
Note that we can without loss of generality assume that $\mathcal{F}_2$ or $\mathcal{F}_3$ has infinite cardinality.\newline

\noindent \textit{Case 1: }$\mathcal{F}_3$ has infinite cardinality.\newline
For every $p\in\mathcal{F}_3$ let 
$$
p^*:=\{\{ a,x\}\in[p]^2\mid a\in p\cap Y\land x\in p\setminus Y\}
$$
and apply $\operatorname{RC}_7$ on the set $X^*:=\bigcup\{p^*\mid p\in\mathcal{F}_3\}$. We get an infinite subset $Y^*\subseteq X^*$ with a choice function $g:[Y^*]^7\to Y^*$. For every $1\leq i\leq 12$ define
$$
\mathcal{F}_i^*:=\{ p^*\mid p\in\mathcal{F}_3\land \lvert p^*\cap Y^*\rvert=i\}.
$$
There is an $1\leq i\leq 12$ such that $\lvert\mathcal{F}_i^*\rvert=\infty$. If $i\notin \{ 6,12\}$ we can choose an element from each $p$ with $p^*\in\mathcal{F}_i^*$ and therefore we are done. If $i=6$, the only case in which we cannot choose an element from all $p$ with $p^*\in\mathcal{F}_6^*$ is the one illustrated in Figure \ref{figure2}:

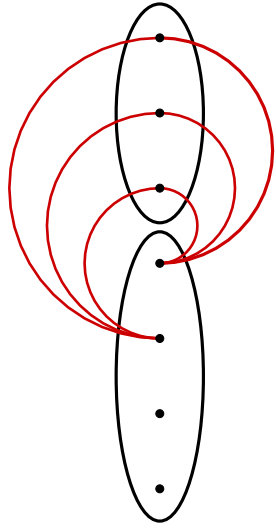
\begin{figure}[h]
\definecolor{cczzqq}{rgb}{0.8,0.6,0.}
\definecolor{qqwwzz}{rgb}{0.,0.4,0.6}
\definecolor{ccqqqq}{rgb}{0.8,0.,0.}
\begin{center}
\begin{tikzpicture}[line cap=round,line join=round,>=triangle 45,x=1.0cm,y=1.0cm]
\clip(0.8,1.4) rectangle (4.7,8.6);
\draw [rotate around={90.:(3.,6.995)},line width=1.2pt] (3.,6.995) ellipse (1.45555108793715cm and 0.5800034220546574cm);
\draw [rotate around={90.:(3.,3.495)},line width=1.2pt] (3.,3.495) ellipse (1.9244857593297757cm and 0.5800176185799213cm);
\draw [shift={(3.,6.5)},line width=1.2pt,color=ccqqqq]  plot[domain=-1.5707963267948966:1.5707963267948966,variable=\t]({1.*1.5*cos(\t r)+0.*1.5*sin(\t r)},{0.*1.5*cos(\t r)+1.*1.5*sin(\t r)});
\draw [shift={(3.,6.)},line width=1pt,color=ccqqqq]  plot[domain=-1.5707963267948966:1.5707963267948966,variable=\t]({1.*1.*cos(\t r)+0.*1.*sin(\t r)},{0.*1.*cos(\t r)+1.*1.*sin(\t r)});
\draw [shift={(3.,5.5)},line width=1pt,color=ccqqqq]  plot[domain=-1.5707963267948966:1.5707963267948966,variable=\t]({1.*0.5*cos(\t r)+0.*0.5*sin(\t r)},{0.*0.5*cos(\t r)+1.*0.5*sin(\t r)});
\draw [shift={(3.,5.)},line width=1pt,color=ccqqqq]  plot[domain=1.5707963267948966:4.71238898038469,variable=\t]({1.*1.*cos(\t r)+0.*1.*sin(\t r)},{0.*1.*cos(\t r)+1.*1.*sin(\t r)});
\draw [shift={(3.,5.5)},line width=1pt,color=ccqqqq]  plot[domain=1.5707963267948966:4.71238898038469,variable=\t]({1.*1.5*cos(\t r)+0.*1.5*sin(\t r)},{0.*1.5*cos(\t r)+1.*1.5*sin(\t r)});
\draw [shift={(3.,6.)},line width=1pt,color=ccqqqq]  plot[domain=1.5707963267948966:4.71238898038469,variable=\t]({1.*2.*cos(\t r)+0.*2.*sin(\t r)},{0.*2.*cos(\t r)+1.*2.*sin(\t r)});
\begin{scriptsize}
\draw [fill=black] (3.,8.) circle (1.5pt);
\draw [fill=black] (3.,7.) circle (1.5pt);
\draw [fill=black] (3.,6.) circle (1.5pt);
\draw [fill=black] (3.,5.) circle (1.5pt);
\draw [fill=black] (3.,4.) circle (1.5pt);
\draw [fill=black] (3.,3.) circle (1.5pt);
\draw [fill=black] (3.,2.) circle (1.5pt);
\end{scriptsize}
\end{tikzpicture}
\end{center}
\caption{Case $i=6$}
\label{figure2}
\end{figure}
But in this case we are done by Lemma \ref{lemma1}. And if $i=12$ we are done by Lemma \ref{lemma2}. \newline

\noindent \textit{Case 2: }$\mathcal{F}_2$ has infinite cardinality.\newline
For every $1\leq i\leq 10$ we define $\mathcal{F}_i^*$ as in Case 1. The only $i$ for which we cannot choose one element from each $p$ with $p^*\in\mathcal{F}_i^*$ or for which we cannot choose three elements from each $p$ with $p^*\in\mathcal{F}_i^*$ in order to reduce it to Case 1, is $i=10$. But in this case we are done by Lemma \ref{lemma3}.
\end{proof}

\section{When does $\operatorname{RC}_{m}\Rightarrow\operatorname{WOC}_{n}^-$ hold?}

\begin{lem}
\label{partition}
Let $m,r,s\in\omega\setminus\{0\},$ let $p_0,\dots, p_{r-1}$ be pairwise different prime numbers and $n_0,\dots, n_{s-1}\in\omega\setminus\{0\}$ with 
$$
\forall i< r~\exists j< s~(p_i\mid n_j).
$$
Moreover assume that there are $a_0,\dots, a_{r-1}\in\omega$ with 
$$
m=\sum_{i< r}a_ip_i.
$$ 
Define $n=\sum_{j\in s}n_j$. Let $\mathcal{F}$ be an infinite, well-ordered family of pairwise disjoint $n$-element sets. We assume that each $F\in\mathcal{F}$ is partitioned into sets $F_i$, $i< s$, with
$$
\lvert F_i\rvert=n_i.
$$
If $\operatorname{RC}_m$ holds, there is an infinite subset $\mathcal{G}\subseteq\mathcal{F}$ such that for each $G\in\mathcal{G}$ we can refine the partition on $G$.
\end{lem}

\begin{proof}
For every $F\in\mathcal{F}$ we define
$$
E_F:=\{ x\in [F]^s\mid \lvert x\cap F_j\rvert=1\text{ for all }j< s\}.
$$
Apply $\operatorname{RC}_m$ to the set $\bigcup_{F\in\mathcal{F}}E_F$. Then there is an infinite subset $Y\subseteq\bigcup_{F\in\mathcal{F}}E_F$ with a choice function 
$$
f:[Y]^m\to Y.
$$
For all $F\in\mathcal{F}$ let
$$
D_F:=E_F\cap Y.
$$
Define $\mathcal{G}:=\{ F\in\mathcal{F}\mid D_F\not = \emptyset\}.$\newline

\noindent \textbf{Claim: }Let $\mathcal{H}\subseteq\mathcal{G}$ be an infinite subset, and let $l=\sum_{i< r}b_ip_i$ with $b_i\leq a_i$ for all $i< r$. Let $p=p_{i_0}\in\{ p_i\mid i< r\}$ such that $b_{i_0}\not =0$. For all $H\in\mathcal{H}$ let $\emptyset\not =C_H\subseteq E_H$. For all $H\in\mathcal{H}$ let $\{ C_H^k\mid k< t_H\}$ with $t_H\in\omega\setminus\{ 0\}$ be a partition of $[C_H]^p$. We assume that there is a choice function 
$$
h:\left[\bigcup_{H\in\mathcal{H}}C_H\right]^{l+p}\to \bigcup_{H\in\mathcal{H}} C_H
$$
with $0\leq l\leq m-p$. Then there is an infinite subset $\mathcal{I}\subseteq\mathcal{H}$ such that
\begin{enumerate}
	\item there is a choice function 
$$
i:\left [\bigcup_{I\in\mathcal{I}} C_I\right ]^{l-p}\to\bigcup_{I\in\mathcal{I}}C_I
$$
or
	\item we can refine the partition $\{I_i\mid i< s\}$ on each $I\in\mathcal{I}$ or
	\item for each $I\in\mathcal{I}$ we can refine the partition $\{ C_I^k\mid k< t_I\}$ on $[C_I]^p$ or
	\item for each $I\in\mathcal{I}$ we can choose 
$$
\emptyset\not =B_I\subsetneq C_I.
$$
\end{enumerate}
Moreover, Case 1 only occurs when $l-p>p_i$ for an $i< r$.
\begin{proof}[Proof of the Claim]
Assume that there is an infinite subset $\mathcal{I}\subseteq\mathcal{H}$ such that there is a $j_0< s$ with 
$$
n_{j_0}\nmid \lvert C_I\rvert
$$
for all $I\in \mathcal{I}$. Let $I\in\mathcal{I}$. For all $z\in I$ define 
$$
\#z:=\lvert\{ x\in C_I\mid z\in x\}\rvert.
$$
Since $\sum_{z\in I_{j_0}}\#z=\lvert C_I\rvert$, it follows that 
$$
\emptyset\not =\{z\in I_{j_0}\mid \forall z^\prime \in I_{j_0} (\#z\leq \#z^\prime))\}\subsetneq I_{j_0}.
$$
Therefore we can refine the partition on each $I\in\mathcal{I}$ and we are done. So we can assume that for all $j< s$ and all $H\in\mathcal{H}$ we have that
$$
n_j\mid \lvert C_I\rvert.
$$ 
There are three cases:\newline

\noindent \textit{Case 1: }There is a $Z_0\in \left [\bigcup_{H\in\mathcal{H}} C_H\right]^l$ such that there is an infinite subset $\mathcal{I}\subseteq\mathcal{H}$ with
$$
\forall I\in\mathcal{I}~\forall x\in [C_I]^p~(h(Z_0\cup x)\in x).
$$
Without loss of generality assume that $Z_0\cap\bigcup_{I\in\mathcal{I}} C_I =\emptyset$. For every $I\in\mathcal{I}$ and all $z\in C_I$ define
$$
\operatorname{deg}_I(z):=\lvert \left\{ x\in [C_I]^p\mid h(Z_0\cup x)=z\right\}\rvert.
$$
Note that $\sum_{z\in C_I}\operatorname{deg}_I(z)=\lvert[ C_I]^p\rvert=\begin{pmatrix}\lvert C_I\rvert \\ p\end{pmatrix}$. We have that $p\mid \lvert C_I\rvert$ because $n_j\mid \lvert C_I\rvert$ for all $j< s$. By Lemma \ref{binom} it follows that $\lvert C_I\rvert\nmid\begin{pmatrix}\lvert C_I\rvert\\ p\end{pmatrix}$. For each $I$ choose
$$
\emptyset\not =B_I:=\{ z\in C_I\mid \forall z^\prime\in C_I (\operatorname{deg}_I(z)\leq \operatorname{deg}_I(z^\prime))\}\subsetneq C_I.
$$

\noindent \textit{Case 2: }There is a $Z_0\in \left [\bigcup_{H\in\mathcal{H}} C_H\right]^l$ such that there is an infinite subset $\mathcal{I}\subseteq\mathcal{H}$ with
$$
\forall I\in\mathcal{I}~\exists k< t_I~\exists x,x^\prime\in C_I^k ~(h(Z_0\cup x)\in Z_0\land h(Z_0\cup x^\prime)\in x^\prime).
$$
In this case we can refine the partition on $[C_I]^p$ for each $I\in\mathcal{I}$.\newline

\noindent \textit{Case 3: }There is a $Z_0\in \left [\bigcup_{H\in\mathcal{H}} C_H\right]^l$ such that there is an infinite subset $\mathcal{I}\subseteq\mathcal{H}$ with
$$
\forall I\in\mathcal{I}~\exists k< t_I~(\forall x\in C_I^k~h(Z_0\cup x)\in Z_0\land \exists x,x^\prime\in C_I^k~(h(Z_0\cup x)\not = h(Z_0\cup x^\prime))).
$$
In this case we can refine the partition on $[C_I]^p$ for each $I\in\mathcal{I}$.

\noindent \textit{Case 4: } For all $Z\in  \left [\bigcup_{H\in\mathcal{H}} C_H\right]^l$ there are infinitely many $H\in\mathcal{H}$ such that 
\begin{equation}
\label{Bedingung}
\exists k< t_H~\forall x,x^\prime\in C_H^k~(h(Z\cup x)=h(Z\cup x^\prime)\in Z).
\end{equation}
For each $Z\in  \left [\bigcup_{H\in\mathcal{H}} C_H\right]^l$ let $H_Z\in \mathcal{H}$ be the smallest element with respect to the well-ordering in $\mathcal{H}$ such that (\ref{Bedingung}) holds. Then let $k_0< t_{H_Z}$ be smallest possible such that 
$$
\forall x,x^\prime\in C_H^{k_0}(h(Z\cup x)=h(Z\cup x^\prime)\in Z).
$$
We define an function $i: \left [\bigcup_{H\in\mathcal{H}} C_H\right]^l\to  \bigcup_{H\in\mathcal{H}} C_H$ by stipulating $i(Z):=h(Z\cup x)$, where $x\in C_H^{k_0}.$ Note that $i$ is well-defined. So 1. of the Claim is satisfied.
\renewcommand{\qedsymbol}{{\large$\boldsymbol{\dashv}_{\text{Claim}}$}}
\end{proof}
Now we apply the Claim iteratively until we can find an infinite subfamily of $\mathcal{G}$ with a refined partition on each $G\in\mathcal{G}$.
\end{proof}

\begin{prop}
\label{positive}
Let $m,n\in\omega\setminus \{0,1\}$ and assume that for all prime numbers $p_0,\dots , p_{r-1}$, $r\in\omega$, such that there are $a_0,\dots, a_{r-1}\in\omega$ with 
$$
n=\sum_{i< r} a_ip_i,
$$
we can find $b_0,\dots, b_{r-1}\in \omega$ with
$$
m=\sum_{i< r}b_ip_i.
$$
Then we have that
$$
\operatorname{RC}_m\Rightarrow \operatorname{WOC}_n^-.
$$
\end{prop}

\begin{proof}
Let $\mathcal{F}$ be an infinite, well-ordered family of $n$-element sets. Assume that for every $F\in\mathcal{F}$ there is a partition of $F$ into sets $F_0,\dots, F_{s-1}$, for an $s\in\omega$. If there is an $i< s$ with $\lvert F_i\rvert =1$, we are done. Otherwise we can apply Lemma \ref{partition} in order to find an infinite subfamily $\mathcal{G}\subseteq\mathcal{F}$ with a refined partition on each $G\in\mathcal{G}$. We do the same process until there is a part containing exactly one element.
\end{proof}

\begin{cor}
\label{Cor6.3}
For every $n\in\omega$ we have that
$$
\operatorname{RC}_n\Rightarrow \operatorname{WOC}_n^-.
$$
\end{cor}

\section{When does $\operatorname{RC}_m\not \Rightarrow \operatorname{WOC}_n^-$ hold?}

In this section we will show that for all $n,m\in\omega\setminus \{0,1\}$ which do not satisfy the conditions of Proposition \ref{positive} we have that
$$
\operatorname{RC}_m\not\Rightarrow \operatorname{WOC}_n^-
$$
in $\operatorname{ZF}$. In a first step we will construct suitable Fraenkel-Mostowski permutation models to show that $\operatorname{RC}_m\not\Rightarrow \operatorname{WOC}_n^-$  in $\operatorname{ZFA}$. We will then see that both statements, $\operatorname{RC}_m$ and $\operatorname{WOC}_n^-$, are injectively boundable. So by \cite[Theorem 3A3]{Pincus72} the result is transferable into $\operatorname{ZF}$.

Let $p_0$ and $p_1$ be two prime numbers. We start with a ground model $\mathcal{M}_{p_0,p_1}$ of $\operatorname{ZFA}+\operatorname{AC}$ with a set of atoms
$$
\mathcal{A}:=\bigcup \{ A_i\mid i\in \omega\}\cup\bigcup \{ B_j\mid j\in\omega\},
$$
where for all $i,j\in\omega$ the sets $A_i$ and $B_j$ are called blocks. These blocks have the following properties:
\begin{enumerate}
	\item For all $i\in\omega$, $A_i=\{ a_{i,k}\mid k< p_0\}$ and $B_i=\{b_{i,k}\mid k< p_1\}$ with $\lvert A_i\rvert=p_0$ and $\lvert B_i\rvert=p_1$.
	\item The blocks are pairwise disjoint.
\end{enumerate}
For all $i,j\in\omega$ we define a permutation on $\mathcal{A}$ as follows:
\begin{enumerate}
\setcounter{enumi}{2}
	\item For all $i\in \omega$ and all $k< p_0$ let
$$
\alpha_i(a_{i,k}):=\begin{cases}a_{i,k+1} &\text{ if } k< p_0-1\\
a_{i,0} &\text{ otherwise}\end{cases}
$$
and $\alpha_i(a)=a$ for all $a\in\mathcal{A}\setminus A_i$. Analogously for all $j\in \omega$ and all $k< p_1$ let
$$
\beta_j(b_{j,k}):=\begin{cases}b_{j,k+1} &\text{ if } k< p_1-1\\
b_{i,0} &\text{ otherwise}\end{cases}
$$
and $\beta_j(b)=b$ for all $b\in\mathcal{A}\setminus B_j$.
\end{enumerate}
Now we define an abelian group $G$ of permutations of $\mathcal{A}$ by requiring
$$
\phi\in G\iff \phi=\alpha\circ\beta,
$$
where
$$
\alpha=\prod_{i\in\omega}\alpha_i^{k_i}\text{ with } k_i< p_0 \text{ for each }i\in\omega
$$
and
$$
\beta=\prod_{j\in\omega}\beta_j^{l_j}\text{ with }l_j< p_1 \text{ for each }j\in\omega.
$$
Let $\mathcal{F}$ be the normal filter on $G$ generated by the subgroups
$$
\fix_G(E)=\{ \phi\in G\mid \forall a\in E(\phi(a)=a)\}
$$
with $E\in\fin(\mathcal{A})$. Let $\mathcal{V}_{p_0,p_1}$ be the class of all hereditarily symmetric sets.

\begin{remark}
We can also work with $k$ blocks of size $p_0,\dots, p_{k-1}$, where $p_i$ is a prime number for every $i< k$. The corresponding permutation model is denoted by $\mathcal{V}_{p_0,\dots, p_{k-1}}$.
\end{remark}

\begin{defi}
A set $E\in\fin(\mathcal{A})$ is closed if and only if for all $i,j\in\omega$ we have that 
$$
A_i\cap E\not =\emptyset\Rightarrow A_i\subseteq E \text{ and } B_j\cap E\not =\emptyset\Rightarrow B_j\subseteq E.
$$
\end{defi}

\begin{remark}
Let $A$ and $B$ be two blocks in $\{A_i\mid i\in\omega\}\cup \{ B_j\mid j\in\omega\}$. We define
$$
A<B:\iff\begin{cases} A=A_i\land B=B_j \text{ or}\\
A=A_i\land B=A_j\land i<j\text{ or}\\
A=B_i\land B=B_j\land i<j.\end{cases}
$$
Moreover, for distinct closed sets $E=\bigcup\{F_0,\dots F_n\}\in\fin(\mathcal{A})$ and $E^\prime=\bigcup\{ F_0^\prime,\dots, F_m^\prime\}\in\fin(\mathcal{A})$ with blocks $F_0,\dots, F_n,$ $F_0^\prime,\dots, F_m^\prime$ let
$$
E\prec E^\prime:\iff \text{ The $<$-least block in the symmetric difference } \{F_0,\dots, F_n\}\Delta\{F^\prime_0,\dots, F_m^\prime\} \text{ belongs to }E.
$$
This defines a well-ordering on the set of all closed supports.
\end{remark}

\begin{lem}
\label{model1}
Let $n\in\omega\setminus\{0,1\}$ and let $p_0$ and $p_1$ be two prime numbers such that 
$$
n=cp_0+dp_1\not =0
$$
for $c,d\in\omega$. Then we have that
$$
\mathcal{V}_{p_0,p_1}\models \lnot\operatorname{WOC}_n^-.
$$
\end{lem}

\begin{proof}Define
$$
\mathcal{F}:=\{ A_l\cup A_{l+1}\cup\dots\cup A_{l+c-1}\cup B_{l+c}\cup\dots\cup B_{l+c+d-1}\mid l=k(c+d)\text{ for a }k\in\omega\}.
$$
This is an infinite family of $n$-element sets. Since the empty set is a support of $\mathcal{F}$, we have that $\mathcal{F}\in\mathcal{V}_{p_0,p_1}$. Moreover, $\mathcal{F}$ is well-orderable in $\mathcal{V}_{p_0,p_1}$. Assume by contradiction that there is an infinite subset $\mathcal{G}\subseteq\mathcal{F}$ with a choice function
$$
g:\mathcal{G}\to \bigcup\mathcal{G}
$$
in $\mathcal{V}_{p_0,p_1}$. Let $E_g\in\fin(\mathcal{A})$ be a closed support of $g$. Since $E_g$ is finite, there is a $G_0\in \mathcal{G}$ such that $G_0\cap E_g=\emptyset$. Then there is an $i\in\omega$ with 
$$
g(G_0)\in A_i\cup B_i.
$$
Define $\gamma_i:=\alpha_i\circ\beta_i$. We have that
$$
g(\gamma_i(G_0))=g(G_0)\not = \gamma_i(g(G_0)).
$$
So $E_g$ is not a support of $g$ which is a contradiction.
\end{proof}

\begin{lem}
\label{model2}
Let $m\in\omega\setminus\{0,1\}$ and let $p_0,p_1$ be prime numbers such that
$$
m\not=cp_0+dp_1
$$
for all $c,d\in\omega$. Then we have that
$$
\mathcal{V}_{p_0,p_1}\models \operatorname{RC}_m.
$$
\end{lem}

\begin{proof}
Let $x\in\mathcal{V}_{p_0,p_1}$ be an infinite set with closed support $E_x\in\fin(\mathcal{A})$. If there is an $E\in\fin(\mathcal{A})$ such that 
$$
y:=\{ z\in x\mid E\text{ is a support of }z\}
$$
is an infinite set, we are done because $y$ can be well-ordered in $\mathcal{V}_{p_0,p_1}$. So assume that for all $E\in\fin(\mathcal{A})$ there are only finitely many $z\in x$ with support $E$. For every closed support $E\in\fin(\mathcal{A})$ with $E_x\subsetneq E$ define
$$
M_E:=\{ z\in x\mid E \text{ is the minimal closed support of $z$ with } E_x\subseteq E\}.
$$
The sets $M_E$ are finite. For all $z\in M_E$ define
$$
[z]:=\{ \phi(z)\mid \phi\in \operatorname{fix}_G(E_x)\}\subseteq M_E.
$$
There are two cases:\newline

\noindent\emph{\textbf{Case 1: }}For infinitely many $M_E$ there is a $z\in M_E$ with
$$
[z]=M_E.
$$
Let $y:=\bigcup \{M_E\mid E_x\subsetneq E\land\exists z\in M_E(M_E=[z])\}.$ The set $y$ is in $\mathcal{V}_{p_0,p_1}$ because $E_x$ is a support of $y$. Let $t\subseteq y$ with $\lvert t\rvert=m$ and let $E$ be the smallest support such that $M_E\subseteq y$ and $\lvert t\cap M_E\rvert$ is not of the form $k_0p_0+k_1p_1$ with $k_0,k_1\in\omega$. Define $t_{-1}:=t\cap M_E$. Since $E\setminus E_x\not =\emptyset$ there are blocks $A_{i_0},\dots, A_{i_{u-1}},B_{j_u},\dots, B_{j_{u+v-1}}$ with 
$$
E\setminus E_x=\bigcup \{ A_{i_0},A_{i_1}\dots, A_{i_{u-1}},B_{j_u},B_{j_{u+1}},\dots, B_{j_{u+v-1}}\}.
$$
Define
$$
\tilde{G}:=\left\{\prod_{k\in u}\alpha_{i_k}^{\kappa_{i_k}}\circ\prod_{l\in v}\beta_{j_{u+l}}^{\lambda_{j_{u+l}}}\mid \forall k< u~\forall l< v~(\kappa_{i_k}< p_0\land \lambda_{j_{u+l}}< p_1)\right\}.
$$
Let $\phi=\alpha_{i_0}^{\kappa_{i_0}}\circ\dots\circ\alpha_{i_{u-1}}^{\kappa_{i_{u-1}}}\circ\beta_{j_u}^{\lambda_{j_u}}\circ\dots\circ\beta_{j_{u+v-1}}^{\lambda_{j_{u+v-1}}}\in \tilde{G}$. Define
$$
\phi\vert_r:=\kappa_{i_r}\text{ if } r< u\text{ and } \phi\vert_r:=\lambda_{j_r} \text{ if } u\leq r< u+v.
$$
The elements in $\tilde{G}$ can be ordered lexicographically. We call this well-ordering $\leq_{\tilde{G}}$. For all $s,s^\prime< t_{-1}$ and all $r< u+v$ define
$$
\operatorname{dist}_r(\langle s,s^\prime\rangle):=\phi\vert_r,
$$
where $\phi$ is the $\leq_{\tilde{G}}$-smallest element in $\tilde{G}$ with $\phi(s)=s^\prime$.\newline

The rest of the proof can be done as in \cite[Proposition 6.6]{Wurzelpaper}. For the sake of completeness, we will redo it here:

\noindent\textit{Claim 1:} For all $s,s^\prime, s^{\prime\prime}< t_{-1}$ and all $r< u+v$ we have that
$$
\operatorname{dist}_r(\langle s,s^\prime\rangle)+_p \operatorname{dist}_r(\langle s^\prime, s^{\prime\prime}\rangle)=\operatorname{dist}_r(\langle s,s^{\prime\prime}\rangle),
$$
where $p=p_0$ if $r< u$ and $p=p_1$ if $u\leq r<u+v$. Moreover, $+_p$ denotes addition modulo $p$.

\begin{proof}[Proof of Claim 1]
Let $\phi_0,\phi_1,\phi\in \tilde{G}$ be $\leq_{\tilde{G}}$-minimal with 
$$
\phi_0(s)=s^\prime, \phi_1(s^\prime)=s^{\prime\prime}\text{ and }\phi(s)=s^{\prime\prime}.
$$
Assume that $\phi\not =\phi_1\circ\phi_0$. So we have that $\phi^{-1}\circ\phi_1\circ \phi_0\not = \operatorname{id}$ and
$$
\phi^{-1}\circ\phi_1\circ\phi_0(s)=s.
$$
Let $l< u+v$ be the largest number such that 
$$
\phi^{-1}\circ\phi_1\circ\phi_0\vert_{l}\not =0.
$$
Without loss of generality we assume that $l< u$. Then let $m\in\omega$ with
$$
(\phi^{-1}\circ\phi_1\circ\phi_0)^m\vert_l=1.
$$
Note that $(\phi^{-1}\circ\phi_1\circ\phi_0)^m\not =\alpha_{i_l}$ because otherwise we would have that $\alpha_{i_l}(s)=s$ which is a contradiction to the fact that $E$ is the minimal support of $s$ with $E_x\subseteq E$.  So there is a $\varphi\in\tilde{G}\setminus\{ \operatorname{id}\}$ with
$$
(\phi^{-1}\circ\phi_1\circ\phi_0)^m=\varphi\circ\alpha_{i_l}\text{ and } \varphi<_{\tilde{G}}\alpha_{i_l}.
$$
Then $\varphi\circ \alpha_{i_l}(s)=s\Rightarrow \alpha_{i_l}(s)=\varphi^{-1}(s)$. Note that $\varphi^{-1}<_{\tilde{G}}\alpha_{i_l}$. We have that $\phi_0\vert_{l}\not =0$ or $\phi_1\vert_{l}\not =0$ or $\phi\vert_{l}\not = 0$. Without loss of generality we assume that $\phi_0\vert_{l}\not =0$. Then
$$
 \phi_0\circ\alpha_{i_l}^{-1}\circ\varphi^{-1}<_{\tilde{G}} \phi_0
$$
and 
$$
\phi_0\circ\alpha_{i_l}^{-1}\circ\varphi^{-1}(s)=\phi_0\circ\alpha_{i_l}^{-1}\circ\alpha_{i_l}(s)=\phi_0(s)=s^\prime.
$$
This contradicts the minimality of $\phi_0$.
\renewcommand{\qedsymbol}{{\large$\boldsymbol{\dashv}_{\text{Claim 1}}$}}
\end{proof}

For all $\tilde{t}\subseteq t_{-1}$, all $s< \tilde{t}$ and all $r< u+v$ define
$$
\chi_r(s,\tilde{t}):=\{ \operatorname{dist}_r(\langle s,s^\prime\rangle)\mid s^\prime\in\tilde{t}\}.
$$
These sets have the following properties:\newline

\noindent\textit{Claim 2:} For all $\tilde{t}\subseteq t_{-1}$ and all $s,s^\prime<\tilde{t}$ we have that
\begin{enumerate}
	\item $1\leq \lvert \chi_r(s,\tilde{t})\rvert\leq p_0$ for all $r< u$ and $1\leq\lvert \chi_r(s,\tilde{t})\rvert\leq p_1$ for all $u\leq r< u+v$.
	\item for all $r< u+v$ there is a $k_r\in \omega$ such that $\chi_r(s,\tilde{t})=\chi_r(s^\prime, \tilde{t})+_p k_r$, where $p=p_0$ if $r< u$ and $p=p_1$ if $u\leq r< u+v$.
	\item $\lvert \chi_r(s,\tilde{t})\rvert=\lvert\chi_r(s^\prime,\tilde{t})\rvert$.
	\item if $s\not = s^\prime$ there is an $r< u+v$ such that $\chi_r(s,\tilde{t})\not =\chi_r(s^\prime, \tilde{t})$.
\end{enumerate}

\begin{proof}[Proof of Claim 2]
1. Note that $0<\chi_r(s,\tilde{t})$ since $\operatorname{dist}_r(\langle s, s\rangle)=0$.\newline
2. Set $k_r:=\phi\vert_r$, where $\phi$ is $\leq_{\tilde{G}}$-minimal with $\phi(s)=s^\prime$ and use Claim 1.\newline
3. This follows from 2.\newline
4. Let $s,s^\prime<\tilde{t}$ and let $\phi$ be $\leq_{\tilde{G}}$-minimal with $\phi(s)=s^\prime$. If $\chi_r(s,\tilde{t})=\chi_r(s^\prime,\tilde{t})$ for all $r< u+v$ it follows that $\phi\vert_r=k_r=0$ for all $r< u+v$. So $\phi=\operatorname{id}$ and therefore $s=s^\prime$.
\renewcommand{\qedsymbol}{{\large$\boldsymbol{\dashv}_{\text{Claim 2}}$}}
\end{proof}

We define an ordering $\preceq$ on the sets $\chi_r(s,\tilde{t})$ as follows: $\chi_r(s,\tilde{t})\preceq \chi_r(s^\prime,\tilde{t})$ if and only if $\chi_r(s,\tilde{t})=\chi_r(s^\prime,\tilde{t}) \text{ or the smallest integer in the symmetric difference } \chi(s,\tilde{t})\Delta\chi_r(s^\prime,\tilde{t}) \text{ belongs to } \chi_r(s,\tilde{t}).$\newline

\noindent For all non-empty sets $\tilde{t}\subseteq t_{-1},$ all $r< u+v$ and all natural numbers $n$ define $\lambda_{r,n}(\tilde{t})$ as follows: Let $\lambda_{r,0}(\tilde{t}):=\emptyset$ and for every $n\in\omega\setminus \{ 0\}$ let
$$
\lambda_{r,n}(\tilde{t}):=\left\{ s\in\tilde{t}\setminus\bigcup_{i=0}^{n-1}\lambda_{r,i}(\tilde{t})\mid \forall s^\prime\in\tilde{t}\setminus\bigcup_{i=0}^{n-1}\lambda_{r,i}(\tilde{t})\left (\chi_r(s,\tilde{t})\preceq \chi_r(s^\prime,\tilde{t})\right )\right\}.
$$
Note that $\bigcup_{n\in\omega}\lambda_{r,n}(\tilde{t})=\tilde{t}$ and only finitely many $\lambda_{r,n}(\tilde{t})$ are non-empty. Assume that $t_{r-1}$ is defined for an $r< u+v$. Then let
$$
t_r:=\lambda_{r,n_0}(t_{r-1}),
$$
where $n_0\in\omega$ is the smallest natural number such that $\lambda_{r,n_0}(t_{r-1})$ is not of the form
$$
cp_0+dp_1
$$
with $c,d\in\omega$. By Claim 2 $t_{u+v-1}$ is a one-element set. I.e.~there is an $s< t$ with 
$$
t_{u+v-1}=\{ s\}.
$$
So we choose $s$ from $t$. This shows that $\operatorname{RC}_m$ holds in $\mathcal{V}_{p_0,p_1}$.\newline

\noindent\emph{\textbf{Case 2: }}There are infinitely many $M_E$ such that there are $z,z^\prime\in M_E$ with 
$$
[z]\cap [z^\prime]=\emptyset.
$$
Our goal is to reduce this case to Case 1. For every $E\in\fin(\mathcal{A})$ with $E_x\subsetneq E$ define $[M_E]:=\{ [z]\mid z\in M_E\}.$ Furthermore choose a $w_0$ in the ground model $\mathcal{M}_{p_0,p_1}\models \operatorname{ZFA}+\operatorname{AC}$ such that 
$$
w_0\setminus\bigcup_{E\in\fin(\mathcal{A}); E_x\subsetneq E}[M_E]=\emptyset\text{ and for all }E\in\fin(\mathcal{A}) \text{ with } E_x\subsetneq E \text{ and } M_E\not =\emptyset~(\lvert w_0\cap [M_E]\rvert =1).
$$
In other words, $w_0$ picks exactly one element from each non-empty $[M_E]$. Note that $E_x$ is a support of $w_0$. So $w_0\in\mathcal{V}_{p_0,p_1}$. Choose
$$
M_E^\prime:=M_E\cap w_0.
$$
This reduces Case 2 to Case 1.
\end{proof}

\begin{prop}
\label{propneg}
Let $m,n\in\omega\setminus \{0,1\}$ and let $p_0,\dots, p_{k-1}$ be $k\in\omega$ prime numbers such that 
$$
m\not = \sum_{i< k}c_ip_i
$$
for all $c_i\in\omega$, $i< k$, and
$$
n=\sum_{i< k}d_ip_i
$$
for some $d_i\in\omega,~i\in k$. Then
$$
\operatorname{RC}_m\not \Rightarrow \operatorname{WOC}_n^-
$$
in $\operatorname{ZF}$.
\end{prop}
\begin{proof}
As in Lemma \ref{model1} and Lemma \ref{model2} we can prove that
\begin{equation}
\label{transfer}
\mathcal{V}_{p_0,\dots, p_{k-1}}\models \operatorname{RC}_m\land \lnot \operatorname{WOC}_n^-.
\end{equation}
In order to transfer this statement to $\operatorname{ZF}$, we have to show that $\operatorname{RC}_n$ and $\operatorname{WOC}_n^-$ are injectively boundable for for all $n\in\omega$. Then we can use Pincus' transfer theorem \cite[Theorem 3A3]{Pincus72}. The terms ``boundable" and ``injectively boundable" are defined in \cite{Pincus72}.\newline

\noindent For a set $x$ we define the injective cardinality
$$
\lvert x\rvert_-:=\{\alpha\in \Omega\mid \text{ there is an injection from $\alpha$ to $x$}\},
$$
where $\Omega$ is the class of all ordinal numbers. Moreover let $\varphi(x)$ denote the property ``if $x$ is an infinite set, there is an infinite $y\subseteq x$ with a choice function on $[y]^n$".  Note that $\varphi(x)$ is boundable. Since $\varphi(x)$ holds when $\lvert x\rvert_->\omega$, it follows that 
$$
\operatorname{RC}_n\iff \forall x(\lvert x\rvert_-\leq\omega\Rightarrow \varphi(x)).
$$
So $\operatorname{RC}_n$ is injectively boundable. Note that $\lnot\operatorname{WOC}_n^-$ is boundable. So (\ref{transfer}) is transferable into $\operatorname{ZF}$.
\end{proof}

Since $\lnot\operatorname{WOC}_n^-\Rightarrow \lnot\operatorname{RC}_n$, Proposition \ref{propneg} gives us:

\begin{cor}
\label{spezialfall}
Let $m,n\in\omega\setminus \{0,1\}$ and let $p_0,\dots, p_{k-1}$ be $k\in\omega$ prime numbers such that 
$$
m\not = \sum_{i< k}c_ip_i
$$
for all $c_i\in\omega$, $i< k$, and
$$
n=\sum_{i< k}d_ip_i
$$
for some $d_i\in\omega,~i< k$. Then
$$
\operatorname{RC}_m\not \Rightarrow \operatorname{RC}_n
$$
in $\operatorname{ZF}$.
\end{cor}

%

\begin{cor}
\label{Cor8.4}
Let $p$ be a prime number, let $m\in\omega\setminus \{0\}$ and $n\in\omega\setminus \{0,1\}$. Then we have that
$$
\operatorname{RC}_{p^m}\Rightarrow \operatorname{WOC}_n^-
$$
if and only if $n\mid p^m$ or $p=2$, $m=1$ and $n=4$.
\end{cor}

\begin{proof}
If $n$ is divisible by a prime $q\not =p$ we have that 
$$
\mathcal{V}_q\models \operatorname{RC}_{p^m}\land\lnot \operatorname{WOC}_n^-.
$$
Therefore, $\operatorname{RC}_{p^m}\not\Rightarrow\operatorname{WOC}_n^-$ in $\operatorname{ZF}$. So we can assume that $n=p^k$ for a $k\in\omega\setminus\{0\}.$\newline

\noindent\textit{Case 1: }$m\geq k$\newline
\noindent Let $p_0,p_1\dots, p_{r-1}$ be $r\in\omega$ prime numbers such that there are $a_0,a_1,\dots, a_{r-1}\in\omega$ with
$$
n=p^k=\sum_{i< r}a_ip_i.
$$
Then 
$$
p^m=p^{m-k}p^k=\sum_{i< r}p^{m-k}a_ip_i.
$$
So by Proposition \ref{positive} we have that 
$$
\operatorname{RC}_{p^m}\Rightarrow \operatorname{WOC}_n^-.
$$

\noindent\textit{Case 2: }$m<k$\newline
\noindent First of all assume that $p\not =2$. By Bertrand's postulate there is a prime number $q_0$ with 
$$
p^m<q_0<2p^m.
$$
Note that $p^k-q_0>p^k-2p^{m}\geq p$ and $q_0\not =p$. So there is a prime number $q_1\not = p$ with
$$
q_1\mid (p^k-q_0).
$$
By construction, $p^k$ can be written as a sum of multiples of $q_0$ and $q_1$. Since $q_1\nmid p^m$ and $p^m<q_0$, we have that
$$
p^m\not = aq_0+bq_1.
$$
for all $a,b\in\omega$. So by Proposition \ref{propneg} we have that
$$
\operatorname{RC}_{p^m}\not \Rightarrow\operatorname{WOC}_{p^k}^-.
$$
Now let $p=2$ and $k\geq 3$. Then there is a prime number $q_0$ with 
$$
2^{k-1}-1<q_0<2^k-2.
$$
It follows that
$$
2<2^{k-1}<q_0<2^k-2.
$$
So $2^k-q_0>2$ and with the same argumentation as above we see that 
$$
\operatorname{RC}_{2^n}\not\Rightarrow\operatorname{WOC}_{2^k}.
$$
Now we assume that $p=2$ and $k=2$. This is the only remaining case. By Proposition \ref{positive} we have that 
$$
\operatorname{RC}_2\Rightarrow \operatorname{WOC}_4^-.
$$
if $m\not =1$.
\end{proof}

\end{document}